\documentclass[11pt,dvips,twoside,letterpaper]{article} 
\usepackage{amssymb,amsthm,amsmath,amscd,amsfonts,txfonts}
\usepackage{xcolor}

\newtheorem{remark}{Remark}[section]

\theoremstyle{plain}

\newtheorem{proposition}{Proposition}[section]
\newtheorem{theorem}{Theorem}[section]
\newtheorem{definition}{Definition}[section]
\newtheorem{corollary}{Corollary}[section]

\newtheorem{example}{Example}[section]

\title{\bfseries\scshape{Color Hom-Akivis algebras, Color Hom-Leibniz algebras  and Modules over Color Hom-Leibniz algebras}}
\author{\bfseries\scshape  Ibrahima BAKAYOKO \thanks{E-mail address: \tt{ibrahimabakayoko@yahoo.fr}}\\
D\'epartement de Math\'ematiques,\\ 
UJNK/Centre Universitaire de N'Z\'er\'ekor\'e,\\
BP : 50, N'Z\'er\'ekor\'e, Guinea.\\
\\\bfseries\scshape Momo BANGOURA\thanks{E-mail address: \tt{bangm59@yahoo.fr}} \\
D\'epartement de Math\'ematiques,\\
Universit\'e Gamal Abdel Nasser de Conakry,\\ 
BP : 1147, Conakry, Guinea. \\
\\\bfseries\scshape Bakary MANGA\thanks{E-mail address: \tt{bakary.manga@ucad.edu.sn, bakary.manga@imsp-uac.org}}\\
D\'epartement de Math\'ematiques et Informatique, \\ 
 Universit\'e Cheikh Anta Diop de Dakar,\\
BP 5005 Dakar-Fann, Dakar, Senegal
\\ and\\
Institut de Math\'ematiques et de Sciences Physiques (IMSP)
\\
 01 BP 613, Porto-Novo, Benin.}

\date{} 

\begin{document} 
\maketitle
%
%
%
%
%
%
%
%
%


\begin{abstract} 
In this paper we introduce color Hom-Akivis algebras and prove that the commutator of any color non-associative Hom-algebra structure 
map leads to a color Hom-akivis algebra. We give various constructions of color Hom-Akivis algebras. Next we study flexible and alternative 
color  Hom-Akivis algebras. Likewise color Hom-Akivis algebras, we introduce non-commutative color Hom-Leibniz-Poisson algebras 
and presente several  constructions. Moreover we give the relationship between Hom-dialgebras and Hom-Leibniz-Poisson algebras; 
i.e. a Hom-dialgebras give rise to a Hom-Leibniz-Poisson algebra. Finally we show that twisting a color Hom-Leibniz module 
structure map by a color Hom-Leibniz algebra endomorphism, we get another one.
\end{abstract} 


\section{Introduction}

Hom-algebraic structures appeared first as a generalization of Lie algebras in \cite{NH}  were the authors
studied $q$-deformations of Witt and Virasoro algebras.
Other interesting Hom-type algebraic structures of many classical structures were studied as Hom-associative algebras, Hom-Lie admissible
 algebras  and more general G-Hom-associative algebras (\cite{AS2}), $n$-ary Hom-Nambu-Lie algebras (\cite{FSA}), 
 Hom-Lie admissible Hom-coalgebras and Hom-Hopf algebras (\cite{AM}), Hom-alternative algebras, Hom-Malcev algebras and Hom-Jordan 
 algebras (\cite{DY3}). Hom-algebraic structures were extended to the case of $G$-graded Lie algebras by studying Hom-Lie superalgebras, 
 Hom-Lie admissible superalgebras  in  \cite{MM}, color Hom-Lie algebras (\cite{KA}) and color Hom-Poisson algebras (\cite{IB1}). 
In \cite{LY}, Yuan  presentes some constructions of quadratique color Hom-Lie algebras, and this is used to provide several examples.
 $T^*$-extensions and central extensions of color Hom-Lie algebras and some cohomological characterizations are established (\cite{FIS}).

 Akivis algebras were introduced by M. A. Akivis (\cite{AMA}) as a tool in the study of some aspects of Web geometry and its connection
with loop theory. These algebras were originally called "W-algebras" (\cite{AMA1}). Later, Hofmann  and Strambach (\cite{HS}) introduce
the term "Akivis algebras" for such algebraic objects.

Hom-Akivis algebras are introduced in \cite{NI1}, in which it is showed that the commutator of non-Hom-associative algebras 
lead to Hom-Akivis algebras. It is also proved that Hom-Akivis algebras can be obtain from Akivis algebras 
by twisting along algebra endomorphisms, and that the class of Hom-Akivis algebras is closed under self-morphisms. 
The connection between Hom-Akivis algebras and Hom-alternative algebras is given.

A non-commutative version of Lie algebras  were introduced by Loday in \cite{JL}. The bracket of the so-called Leibniz algebras
satisfies the Leibniz identity; which, combined with skew-symmetry, is a variation of the Jacobi identity. 
hence Lie algebras are skew-symmetric Leibniz algebras.
In the Leibniz setting, the objects that play the role of associative algebras are called dialgebras, which were
introduced by Loday in \cite{JL1}. Dialgebras have two associative binary operations that
satisfy three additional associative-type axioms.
The relationship between dialgebras and Leibniz algebras are analyzed in \cite{JMC}.
It  is proved in \cite{SMR} that from any Leibniz algebra $L$ one can construct a Leibniz-Poisson algebra $A$ and the properties of $L$
are close to the properties of $A$.  Leibniz algebras were extended to Hom-Leibniz algebras in \cite{AS1}. 
A (co)homology theory of Hom-Leibniz algebras and an initial
study of universal central extensions of Hom-Leibniz algebras was given in \cite{YY}.
The representation of Hom-Leibniz algebras are introduced in \cite{JIR} and the homology of Hom-Leibniz algebras are computed. 
Recently Leibniz superalgebras were studied in \cite{KSM}.
The Hom-dialgebras are introduced in \cite{DY1} as an  extention of 
some work of Loday (\cite{JL1}). In a Hom-dialgebra, there are two binary operations that satisfy five
 $\alpha$-twisted associative-type axioms. Each Hom-Lie algebra can be thought of as a Hom-Leibniz algebra. Likewise, to
every Hom-associative algebra is associated a Hom-dialgebra in which both binary
operations are equal to the original one.

The purpose of this paper is to study color Hom-Akivis algebras (graded version of Hom-Akivis algebras (\cite{NI1})), color 
non-commutative-Hom-Leibniz-Poisson algebras and modules over color Hom-Leibniz algebras.

Section 2 is dedicated to the background material on graded vector spaces, bicharacter and non-Hom-associative algebras.
In section 3, we define color Hom-Akivis algebras which is the graded version of Hom-Akivis algebras. 
Various properties of color Hom-Akivis algebras are studied. More precisely, we show that the commutator 
of color non-Hom-associative algebra structure map lead  to color Hom-Akivis algebras (Theorem \ref{n3}). 
Next, we provide some constructions of color Hom-Akivis algebras; 
on the one hand from color Akivis algebras (Proposition \ref{n5}) and on the other hand from a given color Hom-Akivis algebras 
(Theorem \ref{n4}). Flexible color Hom-Akivis algebras and Hom-alternative color Hom-Akivis algebras are also exposed (Theorem \ref{n6}).
Section 4 being devoted to the study of Color Hom-Leibniz algebras, we define in it color Hom-Leibniz algebras 
and give basic properties (Proposition \ref{n1}). Next, we introduce non-commutative
 Hom-Leibniz-Poisson algebras (in brief color NHLP-algebras). As in the previous section on color Hom-Akivis algebras, 
 we give several twisting of non-commutative Hom-Leibniz-Poisson algebras (Theorem \ref{n1}).
Finally we give the relationship between Hom-dialgebras and Hom-Leibniz-Poisson algebras (Theorem \ref{diax}).
In section 5, we prove that the Yau's twisting of module Hom-algebras (\cite[ Lemma 2.5]{DY2}) works for modules over  color
 Hom-Leibniz algebras  (Theorem \ref{mcl}) likewise it runs for modules over Hom-Lie algebras (\cite{BI}),
 modules over color Hom-Lie  algebras (\cite{IB1}) and modules over color Hom-Poisson algebras (\cite{IB1}).

Throughout this paper, $\bf K$ denotes a field of characteristic $0$ and $G$ is an abelian group. 

\section{Preliminaries} 
\begin{definition}
 A vector space $V$ is said to be a $G$-graded if, there exist a family $(V_a)_{a\in G}$ 
of vector subspaces of $V$ such that $\displaystyle V=\bigoplus_{a\in G} V_a$. 
An element $x\in V$ is said to be homogeneous of degree $a\in G$ if $x\in V_a$. We denote $\mathcal{H}(V)$ the set of all homogeneous 
elements in $V$. 
\end{definition}
\begin{definition}
 Let $\displaystyle V=\bigoplus_{a\in G} V_a$ and $\displaystyle V'=\bigoplus_{a\in G} V'_a$ be two $G$-graded 
vector spaces, and $b\in G$. A linear map $f : V\to V'$ is said to be homogeneous of degree $b$ if 
$f(V_a)\subseteq  V'_{a+b}$, for all $a\in G$. If, $f$ is homogeneous of degree zero i.e. $f(V_a)\subseteq V'_{a}$ 
holds for any $a\in G$, then $f$ is said to be even.
\end{definition}
\begin{definition}
 An algebra $(A, \mu)$ is said to be $G$-graded if its underlying vector space is $G$-graded  and if, furthermore, 
$\mu(A_a, A_b)\subseteq A_{a+b}$, for all $a, b\in G$. Let $A'$ be another $G$-graded algebra. A morphism $f : A\rightarrow A'$ 
of $G$-graded algebras is by definition an algebra morphism from $A$ to $A'$ which is, in addition an even mapping.
\end{definition}
\begin{definition}
 Let $G$ be an abelian group. A map $\varepsilon :G\times G\rightarrow {\bf K^*}$ is called a skew-symmetric bicharacter on $G$ 
 if the following identities hold: for all $a, b, c\in G$, 
\begin{enumerate} 
 \item [(i)] $\varepsilon(a, b)\varepsilon(b, a)=1$;
\item [(ii)] $\varepsilon(a, b+c)=\varepsilon(a, b)\varepsilon(a, c)$;
\item [(iii)]$\varepsilon(a+b, c)=\varepsilon(a, c)\varepsilon(b, c)$.
\end{enumerate}

\end{definition}
%

%
\begin{definition}
\begin{enumerate}
 \item  A multiplicative color non-Hom-associative (i.e. non necessarily Hom-associative) algebra is a quadruple 
 $(A, \mu, \varepsilon, \alpha_A)$ such that
\begin{enumerate}
 \item $A$ is $G$-graded vector space;
\item $\mu : A\otimes A\rightarrow A$ is an even bilinear map;
\item $\varepsilon : G\otimes G\rightarrow K^*$ is a bicharacter;
\item $\alpha_A$ is an endomorphism of $(A,\mu)$ (multiplicativity).
\end{enumerate}
\item $(A, \mu, \alpha)$ is said to be $\varepsilon$-skew-symmetric (resp. $\varepsilon$-commutative) if, for any $x, y\in A$, 
$\mu(x, y)=-\varepsilon(x, y)\mu(y, x)$ (resp. $\mu(x, y)=\varepsilon(x, y)\mu(y, x)$).
\item
Let $(A, \mu, \alpha_A)$ be color non-Hom-associative algebra. For any $x, y, z\in A$, the Hom-associator is defined by 
\begin{eqnarray}
 as(x, y, z)=\mu(\mu(x, y), \alpha_A(z))-\mu(\alpha_A, \mu(y, z)).
\end{eqnarray}
Then $(A, \mu, \varepsilon, \alpha_A)$ is said to be
\begin{enumerate}
\item Hom-flexible, if $as(x, y, x)=0$;
\item Hom-alternative, if $as(x, y, z)$ is $\varepsilon$-skew-symmetric in $x, y, z$. 
\end{enumerate}
\end{enumerate}
\end{definition}

\section{Color Hom-akivis algebras}
Hom-Akivis algebras was introduced in \cite{NI1} as a twisted generalization of Akivis algebras. In the following, we generalize some 
results. We prove that the commutator of a color non-Hom-associative algebra leads to a color Hom-Akivis algebra. Some  properties
of color Hom-Akivis algebras are studied. 
\begin{definition}\label{hai}
 A color Hom-Akivis algebra is quintuple $(A, [-, -], [-, -, -], \varepsilon, \alpha_A)$ consisting of a $G$-graded vector space $A$,
 an even $\varepsilon$-skew-symmetric bilinear map $[-, -]$, an even trilinear map $[-, -, -]$ and an even endomorphism 
 $\alpha_A : A\rightarrow A$ such that for all $x, y, z\in \mathcal{H}(A)$,
\begin{equation}
 \oint\varepsilon(z, x)[[x, y], \alpha_A(z)]=\oint\varepsilon(z, x)\Big([x, y, z]-\varepsilon(x, y)[y, x, z]\Big). \label{hai1}
\end{equation}
If in addition, $\alpha_A$ is an endomorphism with respect to $[-,-]$ and $[-,-,-]$, the color Hom-Akivis algebra $A$ 
is said to be multiplicative.
\end{definition}
\begin{theorem}\label{n3}
 Let $(A, \mu, \varepsilon, \alpha_A)$ be a color non-Hom-associative algebra. Define the maps $[-, -]: A\otimes A\rightarrow A$ and
$[-, -, -]_{\alpha_A} : A\otimes A\otimes A\rightarrow A$ as follows: for all $x,y,z\in A$, 
\begin{eqnarray}
[x, y] &:=& \mu(x, y)-\varepsilon(x, y)\mu(y, x), \cr 
[x, y, z]_{\alpha_A}&:=&as(x, y, z).\nonumber
\end{eqnarray}
Then $(A, [-, -], [-, -, -]_{\alpha_A}, \varepsilon, \alpha_A)$ is a multiplicative color Hom-Akivis algebra.
\end{theorem}
\begin{proof}
Color Hom-Akivis identity (\ref{hai1}) is proved by direct computation. For any homogeneous elements $x, y, z\in A$, we have
\begin{eqnarray}
 \varepsilon(z, x)[[x, y], \alpha_A(z)]&=&\varepsilon(z, x)(xy)\alpha_A(z)-\varepsilon(y, z)\alpha_A(z)(xy)\nonumber\\
&&-\varepsilon(z, x)\varepsilon(x, y)(yx)\alpha_A(z)+\varepsilon(x, y)\varepsilon(y, z)\alpha_A(z)(yx).\nonumber
\end{eqnarray}
And, 
\begin{eqnarray}
 \varepsilon(z, x)([x, y, z]-\varepsilon(x, y)[y, x, z])&=&\varepsilon(z, x)(xy)\alpha_A(z)-\varepsilon(z, x)\alpha_A(x)(yz)\nonumber\\
&&-\varepsilon(z, x)
\varepsilon(x, y)(yx)\alpha_A(z)+\varepsilon(z, x)\varepsilon(x, y)\alpha_A(y)(xz).\nonumber
\end{eqnarray}
Expanding 
$\oint\varepsilon(z, x)[[x, y], \alpha_A(z)]$ and $\oint\varepsilon(z, x)([x, y, z]-\varepsilon(x, y)[y, x, z])$, one gets (\ref{hai1})
and so the conclusion holds.
The multiplicativity follows from that of $(A, \mu, \varepsilon, \alpha_A)$.
\end{proof}
\begin{definition}
 Let $(A, [-, -], [-, -, -], \varepsilon, \alpha_A)$ and $(\tilde A, \{-, -\}, \{-, -, -\}, \varepsilon, \alpha_{\tilde A})$ be two color
 Hom-Akivis algebras. A morphism $f : A\rightarrow\tilde A$ of color Hom-Akivis algebras is an even linear map of $G$-graded vector spaces
 $A$ and $\tilde A$ such that 
\begin{eqnarray}
 f([x, y])&=&\{f(x), f(y)\},\nonumber\\
f([x, y, z])&=&\{f(x), f(y), f(z)\}.\nonumber
\end{eqnarray}
\end{definition}
For example, if we take $(A, [-, -], [-, -, -], \varepsilon, \alpha_A)$ as a multiplicative color Hom-Akivis algebra, then the twisting self-map 
$\alpha_A$ is itself an endomorphism of $(A, [-, -], [-, -, -], \varepsilon)$.

The following theorem is the color version of Theorem 4.4 in \cite{NI1}. 
\begin{theorem}\label{n4}
 Let $A_{\alpha_A}=(A, [-, -], [-, -, -], \varepsilon, \alpha_A)$ be a color Hom-Akivis algebra and 
 $\beta : A\to A$ an even endomorphism.
Then, for any integer $n\geq 1$, $A_{\beta^n}=(A, [-, -]_{\beta^n}, [-, -, -]_{\beta^{n}}, \varepsilon,\beta^n\circ\alpha_A)$ 
is a color Hom-Akivis algebra, where for all $x, y, z\in \mathcal{H}(A)$,
\begin{eqnarray}
 [x, y]_{\beta^n}&:=&{\beta^n}([x, y]),\cr
[x, y, z]_{\beta^n}&:=&{\beta^{2n}}([x, y, z]). \nonumber
\end{eqnarray}
Moreover, if $A_{\alpha_A}$ is multiplicative and $\beta$ commutes with $\alpha_A$, then $A_{\beta^n}$ is multiplicative.
\end{theorem}
\begin{proof}
It is clear that $[-, -]_{\beta^n}$ and $[-, -, -]_{\beta^n}$ are bilinear and trilinear respectively. It is also clear that
the $\varepsilon$-skew-symmetry of $[-, -]_{\beta^n}$ comes from that of $[-, -]$. It remains to prove the color Hom-Akivis identity
 (\ref{hai1}) and the multiplicativity of $A_{\beta^n}$. For any $x, y, z\in \mathcal{H}(A)$, we have,
  \begin{eqnarray}
   \oint\varepsilon(z, x)[[x, y]_{\beta^n}, (\beta^n\circ\alpha_A)(z)]_{\beta^n}
&=&\oint\varepsilon(z, x)\beta^n[\beta^n[x, y], (\beta^n\circ\alpha_A)(z)]\nonumber\\
&=&\beta^{2n}\oint \varepsilon(z, x)([x, y, z]-\varepsilon(x, y)[y, x, z])\nonumber\\
&=&\oint \varepsilon(z, x)(\beta^{2n}[x, y, z]-\varepsilon(x, y)\beta^{2n}[y, x, z])\nonumber\\
&=&\oint \varepsilon(z, x)([x, y, z]_{\beta^{2n}}-\varepsilon(x, y)[y, x, z]_{\beta^{2n}}).\nonumber
  \end{eqnarray}
The multiplicativity is proved similarly as in \cite[Theorem 4.4]{NI1}.
\end{proof}
\begin{corollary}\label{n2}
 Let $(A, [-, -], [-, -, -], \varepsilon)$ be a color Akivis algebra and $\beta : A\rightarrow A$ an even endomorphism. Then
$(A, [-, -]_{\beta^n}, [-, -, -]_{\beta^{n}}, \varepsilon, \alpha_A)$ is a multiplicative color Hom-Akivis algebra.

Moreover, suppose that $(\tilde A, \{-, -\}, \{-, -, -\}, \varepsilon)$ is another color Akivis algebra and $\tilde\alpha_A$ 
an even endomorphism of $\tilde A$. If $f : A\rightarrow \tilde A$ is a morphism of color Akivis algebras such that 
$f\circ \alpha_A={\alpha_{\tilde A}}\circ f$, then 
$f : (A, [-, -], [-, -, -], \varepsilon, \alpha_A)\rightarrow (\tilde A, \{-, -\}, \{-, -, -\}, \varepsilon, \alpha_{\tilde A})$
is a morphism of multiplicative color Hom-Akivis algebras.
\end{corollary}
\begin{proof}
It is similar to that of  \cite[Corollary 4.5]{NI1}.
\end{proof}
 The above construction of color Hom-Akivis algebras is used in \cite{NI1}, in the case of Hom-Akivis algebras,
 to produce examples of Hom-Akivis algebras. 
\begin{proposition}\label{n5}
 Let $(A, [-, -], [-, -, -], \varepsilon)$ be a color Akivis algebra and $\beta : A\rightarrow A$ an even endomorphism. Define
$[-, -]_{\beta^n}$ and $[-, -, -]_{\beta^{n}}$ by: for all $x, y, z\in \mathcal{H}(A)$, 
\begin{eqnarray}
 [x, y]_{\beta^n}&:=&{\beta}([x, y]_{\beta^{n-1}}),\cr 
[x, y, z]_{\beta^n}&:=&{\beta^{2}}([x, y, z]_{\beta^{n-1}}),\nonumber
\end{eqnarray}
Then $(A, [-, -]_{\beta^n}, [-, -, -]_{\beta^{n}}, \varepsilon, \alpha_A)$ is a multiplicative color Hom-Akivis algebra.
\end{proposition}
\begin{proof}
It is proved recurrently by applying Corollary \ref{n2}. The reader may also see \cite[Theorem 4.8]{NI1} for the proof. 
\end{proof}
\begin{definition}
 A color Hom-Akivis algebra $(A, [-, -], [-, -, -], \varepsilon), \alpha_A$ is said to be
\begin{enumerate}
 \item Hom-flexible, if $[x, y, x]=0$ for all $x, y\in A$ ;
\item Hom-alternative, if $[-, -, -]$ is $\varepsilon$-alternating (i.e. $[-, -, -]$ is $\varepsilon$-skew- symmetric 
for any pair of variables).
\end{enumerate}
\end{definition}
\begin{theorem}\label{n6}
 Let $(A, \mu, \varepsilon, \alpha_A)$ be a color non-Hom-associative algebra and the quintuple 
 $(A, [-, -], as(-, -, -), \varepsilon, \alpha_A)$ 
 its associated  color Hom-Akivis algebra (as in Theorem \ref{n3}). Then we have the following:
\begin{enumerate}
 \item if $(A, \mu, \varepsilon, \alpha_A)$ is Hom-flexible, then $(A, [-, -], as(-, -, -), \varepsilon, \alpha_A)$ is Hom-flexible;
\item if $(A, \mu, \varepsilon, \alpha_A)$ is Hom-alternative, then so is $(A, [-, -], as(-, -, -), \varepsilon, \alpha_A)$.
\end{enumerate}
\end{theorem}
\begin{proposition}
 Let $\mathcal{A}_{\alpha_A}=(A, [-, -], [-, -, -], \varepsilon, \alpha_A)$ be a  Hom-flexible color Hom-Akivis algebra. Then we have
\begin{eqnarray}
 \oint\varepsilon(z, x)=\oint\Big(\varepsilon(z, x)+\varepsilon(x, y)\varepsilon(y, z)\Big)[x, y, z].\label{cvf}
\end{eqnarray}
In particular, $\mathcal{A}_{\alpha_A}$ is a color Hom-Lie algebra if and only if 
\begin{eqnarray}
\oint\Big(\varepsilon(z, x)+\varepsilon(x, y)\varepsilon(y, z)\Big)[x, y, z]=0.
\end{eqnarray}
\end{proposition}
\begin{proof}
By expanding the left hand side of color Hom-Akivis identity (\ref{hai1}), and using the Hom-flexibility we get (\ref{cvf}). 
\end{proof}
For more details on color Hom-Lie algebras, see  \cite{KA}, \cite{IB1}, \cite{LY}.
\section{Color Hom-Leibniz algebras}
 In this section, we write NHLP-algebra (resp. NLP-algebra) for non-commutative Hom-Leibniz-Poisson algebra 
 (resp. non-commutative Leibniz-Poisson algebra). Interrelation between Hom-dialgebras and NHLP-algebras are presented.
 
\subsection{Color NHLP-algebras}
This subsection is devoted to color Hom-Leibniz algebras. Color NHLP-algebras are introduced and their 
various twisting are given as well as examples. 
\begin{definition}
 A color Hom-Leibniz algebra is a quadruple $(L, \mu, \varepsilon, \alpha_L)$, consisting of a graded vector space $L$, an even bilinear map
 $\mu : L\otimes L\to L$, a bicharacter $\varepsilon : G\otimes G\rightarrow {\bf K}^*$ and an even linear space homomorphism 
$\alpha_L : L\rightarrow L$ satisfying 
\begin{eqnarray}
 \mu\Big(\alpha_L(x), \mu(y, z)\Big)=\mu\Big(\mu(x, z), \alpha(y)\Big)+\varepsilon(x, y)\mu\Big(\alpha(x), \mu(y, z)\Big). \label{cla2}
\end{eqnarray}
A morphism of color Hom-Leibniz algebras is an even linear map which preserves the structures.
\end{definition}
\begin{remark}
\begin{enumerate}
\item When $G$ is an abelian group with trivial grading, we obtain the following Hom-Leibniz identity (\cite{NI1}):
\begin{eqnarray}
 \mu\Big(\alpha_L(x), \mu(y, z)\Big)=\mu\Big(\mu(x, y), \alpha(z)\Big)+\mu\Big(\alpha(y), \mu(x, z)\Big). \label{cla}
\end{eqnarray}
 \item  The original definition of Hom-Leibniz algebras (\cite{AS1}) is related to the identity
\begin{eqnarray}
 \mu\Big(\mu(x, y), \alpha_L(z)\Big)=\mu\Big(\alpha(x), \mu(y, z)\Big)+\mu\Big(\mu(x, z), \alpha(y)\Big) \label{cla4}
\end{eqnarray}
which is expressed in terms of (right) adjoint homomorphisms $ad_xy:\mu(x, y)$ of $(A, \mu, \alpha_A)$. This justifies the terms of 
``(right) Hom-Leibniz algebra'' that could be used for the Hom-Leibniz algebras defined in \cite{AS1}. 
The dual of (\ref{cla}) is (\ref{cla4}) and in this paper we consider only left color Hom-Leibniz algebras.
\end{enumerate}
\end{remark}
We have the following result.
\begin{proposition}\label{n7}
Let $(L, \mu, \varepsilon, \alpha_L)$ be a color Hom-Leibniz algebra.  Then
 \begin{eqnarray}
\Big(x\cdot y+\varepsilon(x, y) y\cdot x\Big)\alpha_L(z)=0, \label{ia}
\end{eqnarray}
\begin{eqnarray}
[x\cdot y, \alpha_L(z)]+\varepsilon(x, y)[\alpha_L(y), x\cdot z]=\alpha_L(x)\cdot[y, z],
 \end{eqnarray}
where we have  putted $\mu(x, y)=x\cdot y$ and $[x, y]=x\cdot y-\varepsilon(x, y)y\cdot x$.
\end{proposition}
\proof
 The identity (\ref{cla}) implies that
\begin{eqnarray}
(x\cdot y)\cdot\alpha_L(z) = \alpha_L(x)\cdot(y\cdot z)-\varepsilon(x, y)\alpha_L(y)\cdot(x\cdot z) \nonumber
\end{eqnarray}
Likewise, interchanging $x$ and $y$, we have 
\begin{eqnarray}
(y\cdot x)\cdot\alpha_L(z) = \alpha_L(y)\cdot(x\cdot z)-\varepsilon(y, x)\alpha_L(x)\cdot(y\cdot z)\nonumber
\end{eqnarray}
Then, multiplying the second equality by $\varepsilon(x, y)$, and adding memberwise with the first one, we obtain  (\ref{ia}). Next, by direct
computation, we have
\begin{eqnarray}
[x\cdot y, \alpha_L(z)]+\varepsilon(x, y)[\alpha_L(y), x\cdot z]\!\!\!\!
&=&\!\!\!\!(x\cdot y)\cdot\alpha_L(z)-\varepsilon(x, z)\varepsilon(y, z)\alpha_L(z)\cdot(x\cdot y)\cr 
& &\!\!\!\!+\varepsilon(x, y)\alpha_L(y)\cdot(x\cdot z)-\varepsilon(y, z)(x\cdot z)\cdot\alpha_L(y)\cr 
&=&\!\!\!\!\alpha_L(x)\cdot(y\cdot z)-\varepsilon(x, z)\varepsilon(y, z)((z\cdot x)\cdot\alpha_L(y) \cr 
& &\!\!\!\!+\varepsilon(z, x)\alpha_L(x)\cdot(z\cdot y))-\varepsilon(y, z)(x\cdot z)\cdot\alpha_L(y) \mbox{ (by (\ref{cla2})) }\nonumber\\
&=&\!\!\!\!\alpha_L(x)\cdot(y\cdot z)-\varepsilon(y, z)\alpha_L(x)\cdot(z\cdot y)\quad\mbox{(by}\;(\ref{ia}))\nonumber\\
&=&\!\!\!\!\alpha_L(x)\cdot[y, z]\nonumber. ~~~~~~~~~~~~~~~~~~~~~~~~~~~~~~~~~~~~~~~~~~~~~~~~~~~~~~~~~~\qed
\end{eqnarray}
Now we define color NHLP-algebras which is the graded and Hom-version of NLP-algebras (\cite{JMC}).
\begin{definition} \label{nhlp}
 A color NHLP-algebra is a $G$-graded vector space $P$ together with two even bilinear maps 
$[-, -] : P\otimes P\rightarrow P$ and $\mu : P\otimes P\rightarrow P$, a bicharacter $\varepsilon : G\otimes G\to {\bf K}^*$ and
$\alpha_P : P\rightarrow P$ an even linear map such that, for any $x, y, z\in P$, 
\begin{enumerate}
 \item [i)]
 $(P, [\cdot, \cdot], \varepsilon, \alpha_P)$ is a color Hom-Leibniz-Poisson algebra i.e.
 \begin{eqnarray}
  [\alpha_P(x), [y, z]]=[[x, y], \alpha_P(z)]+\varepsilon(x, y)[\alpha_P(y), [x, z]], \label{cpa}
 \end{eqnarray}
\item [ii)]
$(P, \mu, \varepsilon, \alpha_P)$ is a color Hom-associative algebra i.e.
\begin{eqnarray}
 \mu\Big(\alpha(x),\mu(y,z)\Big)=\mu\Big(\mu(x,y),\alpha(z)\Big) \mbox{ (Hom-associativity) };
\end{eqnarray}
\item [iii)]
 and the following identity holds:
 \begin{eqnarray}
  [\alpha_P(x), \mu(y, z)]=\mu([x, y], \alpha_P(z))+\varepsilon(x, y)\mu(\alpha_P(y), [x, z]). \label{comp}
 \end{eqnarray}
\end{enumerate}
If in addition, $\alpha_P$ is an endomorphism with respect to $\mu$ and $[-, -]$, we say that 
$(P, \mu, [\cdot, \cdot], \varepsilon, \alpha_P)$ is a multiplicative color NHLP-algebra.
\end{definition}
\begin{remark}
When the Hom-associative product $\mu$ is $\varepsilon$-commutative, then $(A, \mu, [-, -]), \varepsilon, \alpha_A)$ is said to be
a commutative color Hom-Leibniz Poisson algebra.
\end{remark}
\begin{example}
\begin{enumerate}
 \item 
Any color Hom-Poisson algebra is a color NHLP-algebra algebra.
\item
Any color Hom-Leibniz algebra is a color NHLP-algebra.
\item
If $P$ is a Leibniz-Poisson algebras (viewed as Hom-Leibniz-Poisson algebras with trivial twisting and trivial grading), then the vector
space $P\otimes P$ is a NHLP-algebra with the operations
$$ (x_1\otimes x_2)(y_1\otimes y_2)= x_1y_1\otimes y_1y_2,$$
$$[x_1\otimes x_2, y_1\otimes y_2]=[[x_1, x_2], y_1]\otimes y_2+y_1\otimes[[x_1, x_2], y_2].$$
\end{enumerate}
\end{example}
The proposition below is a direct consequence of the Definition \ref{nhlp}.
\begin{proposition}
 Let $(P, [\cdot, \cdot], \mu, \varepsilon, \alpha)$ be a color NHLP-algebra. Then
$(P, [\cdot, \cdot], \mu^{op}, \varepsilon, \alpha)$ and $(P, k\mu, k[-, -], \alpha_P)$ are also color NHLP-algebras, with
$\mu^{op}(x, y)=\mu(y, x)$ and $k \in {\bf K}^*$.
\end{proposition}
It is well known in Hom-algebras setting that one can obtain Hom-algebra structures from an ordinary one and an endomorphism. 
The following theorem, which gives  a way to construct a color NHLP-algebras from color NLP-algebras and an endomorphism, 
is a similar result.
\begin{theorem}\label{n1}
 Let $(A, \mu, [-, -], \varepsilon)$ be a color NLP-algebra and $\alpha_A : A\rightarrow A$ an even endomorphism. 
 Then, for any integer $n\geq 1$, $(A, \mu_{\alpha_A^n}, [-, -]_{\alpha_A^{n}}, \varepsilon, \alpha^n_A)$ 
 is a multiplicative color NHLP-algebra, where for all $x, y, z\in \mathcal{H}(A)$, 
\begin{eqnarray}
 \mu_{\alpha_A^n}(x, y)&:=&{\alpha_A^n}(\mu(x, y)),\cr
[x, y]_{\alpha_A^n}(x, y)&:=&{\alpha_A^{n}}([x, y]).\nonumber
\end{eqnarray}
Moreover, suppose that $(\tilde A, \tilde\mu, \{-, -\}, \varepsilon)$ is another color NLP-algebra and $\tilde\alpha_A$ 
an even endomorphism of $\tilde A$. If $f : A\rightarrow \tilde A$ is a morphism of color NLP-algebras such that 
$f\circ \alpha_A={\alpha_{\tilde A}}\circ f$, then 
$f : (A, \mu, [-, -], \varepsilon, \alpha_A)\rightarrow (\tilde A, \tilde\mu, \{-, -, -\}, \varepsilon, \alpha_{\tilde A})$
is a morphism of multiplicative color NHLP-algebras.
\end{theorem}
\begin{proof} 
We need to show that  $(A, \mu_{\alpha_A^n}, [-, -]_{\alpha_A^{n}}, \varepsilon, \alpha^n_A)$ satisfies relations (\ref{cpa}) and 
(\ref{comp}). We have, for any $x, y, z\in A$,
\begin{eqnarray}
[\alpha^n_A(x), \mu_{\alpha_A^n}(y, z)]_{\alpha_A^{n}}
&=&\alpha_A^{n}([\alpha^n_A(x), \alpha_A^n\mu(y, z)])\nonumber\\
&=&\alpha_A^{2n}([x, \mu(y, z)])\nonumber\\
&=&\alpha_A^{2n}([[x, y], z]+\varepsilon(x, y)[y, [x, z]])\nonumber\\
&=&\alpha_A^{n}([\alpha_A^{n}([x, y]), \alpha_A^{n}(z)]+\varepsilon(x, y)[\alpha_A^{n}(y), \alpha_A^{n}([x, z])])\nonumber\\
&=&[[x, y]_{\alpha_A^{n}}, \alpha_A^{n}(z)]_{\alpha_A^{n}}
+\varepsilon(x, y)[\alpha_A^{n}(y), [x, z]]_{\alpha_A^{n}}.\nonumber
\end{eqnarray}
The compatibility condition (\ref{comp}) is pointed out similarly and
the second assertion is proved as in the case of Theorem \ref{n4}.
\end{proof}
\begin{example} The commutative Leibniz-Poisson algebra used in this example is given in \cite{SMR1}, \cite{SMR}.
 Let $(A, [-, -])$ is a Leibniz algebra over a commutative infinite field $\bf K$ and let
$$\tilde A=A\oplus {\bf K}$$
be a vector space with multiplications $\cdot$ and $\{-, -\}$ defined as: for $x, y\in A$ and $a, b\in {\bf K}$, 
$$(x+a)\cdot(y+b)=(bx+ay)+ab\quad  \text{ and } \quad  \{x+a, y+b\}=[x, y].$$
Consider the homomorphism $\alpha_{\tilde A}=\alpha_A\oplus Id_{\bf K}:\tilde A\to\tilde A$, then 
$(\tilde A, \cdot, \{-, -\}, \alpha_{\tilde A})$ is a commutative color Hom-Leibniz-Poisson algebra (with trivial grading).
\end{example}
The next result gives a procedure to produce color NHLP-algebras from given one.
\begin{theorem}
 Let $(P, \mu, [-, -], \varepsilon, \alpha_P)$ be a color  NHLP-algebra and $\beta : P\rightarrow P$ an even endomorphism.
Then $(P, \mu_{\beta^n}, [-, -]_{\beta^{n}}, \varepsilon, \beta^n\circ\alpha_P)$ is a color NHLP-algebra, with
\begin{eqnarray}
 \mu_{\beta^n}(x, y)&:=&{\beta^n}(\mu(x, y)),\nonumber\\
\nonumber
[x, y]_{\beta^n}(x, y)&:=&{\beta^{n}}([x, y]),\nonumber
\end{eqnarray}
for all $x, y, z\in \mathcal{H}(P)$.
\end{theorem}
\proof 
We need to show that  $(P, \mu_{\beta^n}, [-, -]_{\beta^{n}}, \varepsilon, \beta^n\circ\alpha_P)$ satisfies relations (\ref{cpa}) and 
(\ref{comp}). We have, for any $x, y, z\in \mathcal{H}(P)$,
\begin{eqnarray}
 [({\beta^{n}}\circ\alpha_P)(x), [y, z]_{\beta^{n}}]_{\beta^{n}}
&=&{\beta^{n}}[({\beta^{n}}\circ\alpha_P)(x), {\beta^{n}}[y, z]]\nonumber\\
&=&{\beta^{2n}}([\alpha_P(x), [y, z]])\nonumber\\
&=&{\beta^{2n}}([[x, y], \alpha_P(z)]+\varepsilon(x, y)[\alpha_P(y), [x, z]])\nonumber\\
&=&{\beta^{n}}[{\beta^{n}}[x, y], ({\beta^{n}}\circ\alpha_P)(z)]+\varepsilon(x, y)[({\beta^{n}}\circ\alpha_P)(y), {\beta^{n}}[x, z]]\nonumber\\
&=&[[x, y]_{\beta^{n}}, ({\beta^{n}}\circ\alpha_P)(z)]_{\beta^{n}}
+\varepsilon(x, y)[({\beta^{n}}\circ\alpha_P)(y), [x, z]_{\beta^{n}}]_{\beta^{n}}.\nonumber
\end{eqnarray}
Next,
\begin{eqnarray}
 [(\beta^n\circ\alpha_P)(x), \mu_{\beta^n}(y, z)]_{\beta^{n}}
&=&\beta^{n}[(\beta^n\circ\alpha_P)(x), \beta^n\mu(y, z)]\nonumber\\
&=&\beta^{2n}[\alpha_P(x), \mu(y, z)]\nonumber\\
&=&\beta^{2n}([[x, y], \alpha_P(z)]+\varepsilon(x, y)[\alpha_P(y), [x, z]])\nonumber\\
&=&\beta^{n}([\beta^{n}[x, y], (\beta^{n}\circ\alpha_P)(z)]+\varepsilon(x, y)[(\beta^{n}\circ\alpha_P)(y), \beta^{n}[x, z]])\nonumber\\
&=&[[x, y]_{\beta^{n}}, (\beta^{n}\circ\alpha_P)(z)]_{\beta^{n}}
+\varepsilon(x, y)[(\beta^{n}\circ\alpha_P)(y), [x, z]_{\beta^{n}}]_{\beta^{n}}\nonumber. \qed 
\end{eqnarray}
\begin{proposition}
 Let $(P, \mu, [-, -], \varepsilon)$ be a color NLP-algebra and $\beta : P\rightarrow P$ an even endomorphism. Define
$\mu_{\beta^n}$ and $[-, -]_{\beta^{n}}$ by
\begin{eqnarray}
 \mu_{\beta^n}(x, y)&:=&{\beta}(\mu_{\beta^{n-1}}(x, y)),\nonumber\\
\nonumber
[x, y]_{\beta^n}&:=&{\beta^{n}}([x, y]_{\beta^{n-1}}),\nonumber
\end{eqnarray}
for all $x, y, z\in \mathcal{H}(P)$.
Then $(P, \mu_{\beta^n}, [-, -]_{\beta^{n}}, \varepsilon, \alpha_P)$ is a multiplicative color NHLP-algebra.
\end{proposition}
\subsection{NHLP-algebras and Dialgebras}
It is well known that Hom-dialgebras give rise to Hom-Leibniz algebra (\cite{DY1}). 
In this subsection, we give a similar result for NHLP-algebras.
\begin{definition}[\cite{DY1}]\label{dia}
 A Hom-dialgebra is a quadruple $(D, \dashv\,, \vdash, \alpha_D)$, where $D$ is a $\bf K$-vector space, 
 $\dashv\,, \vdash : D\otimes D\rightarrow D$ are
bilinear maps and $\alpha_D : D\rightarrow D$ a linear map such that the following  five axioms are satisfied for $x, y, z\in D$ :
 \begin{eqnarray}
(x\vdash y)\dashv\alpha_D(z)&=&\alpha_D(x)\vdash(y\dashv z), \nonumber\\
  \alpha_D(x)\dashv (y\dashv z)&=&(x\dashv y)\dashv\alpha_D(z)=\alpha_D(x)\dashv(y\vdash z)\nonumber\\
(x\dashv y)\vdash\alpha_D(z)&=&=\alpha_D(x)\vdash(y\vdash z)=(x\vdash y)\vdash\alpha_D(z)\nonumber
 \end{eqnarray}
\end{definition}
We say that a Hom-dialgebra $(D, \dashv\,, \vdash, \alpha_D)$ is Hom-associative if it so for the operations $\vdash$ and $\dashv$.
So any Hom-dialgebra is Hom-associative.

The following result connects Hom-dialgebras and Hom-Leibniz-Poisson algebras.
\begin{theorem}\label{diax}
 Let $(D, \dashv\,, \vdash, \alpha_D)$ be a Hom-dialgebra. Define the bilinear map $[-, -] : D\otimes D\rightarrow D$ by setting
\begin{eqnarray}
 [x, y]:=x\vdash y-y\dashv x.
\end{eqnarray}
Then $(D, \dashv\,, [-, -], \alpha_D)$ is a Hom-Leibniz-Poisson algebra.
\end{theorem}
\begin{proof}
We write down all twelve terms involved in the Hom-Leibniz identity (\ref{cla}) :
\begin{eqnarray}
 [\alpha_D(x), [y, z]]&=&  \alpha_D(x)\vdash(y\vdash z)-(y\vdash z)\dashv\alpha_D(x)
                           -\alpha_D(x)\vdash(z\dashv y)+(z\dashv y)\dashv \alpha_D(x)\cr 
[\alpha_D(y), [x, z]] &=&  \alpha_D(y)\vdash(x\vdash z)-(x\vdash z)\dashv\alpha_D(y)
                           -\alpha_D(y)\vdash(z\dashv x)+(z\dashv x)\dashv \alpha_D(y)\cr 
[[x, y], \alpha_D(z)] &=& (x\vdash y)\vdash \alpha_D(z)-\alpha_D(z)\dashv(x\vdash y)
                          -(y\dashv x)\vdash\alpha_D(z)+\alpha_D(z)\dashv(y\dashv x)\nonumber.
\end{eqnarray}
Using the Hom-dialgebra axioms in Definition \ref{dia}, it is readily seen that (\ref{cla}) holds.
Next,
\begin{eqnarray}
[\alpha_D(x), y\dashv z]-[x, y]\dashv\alpha_D(z)-\alpha_D(y)\dashv[x, z]\!\!\!\! &=&\!\!\!\! \alpha_D(x)\vdash(y\dashv z)
-(y\dashv z)\dashv\alpha_D(x)\cr 
& &-(x\vdash y)\dashv\alpha_D(z)+(y\dashv x)\dashv\alpha_D(z) \cr 
& & -\alpha_D(y)\dashv(x\vdash z)+\alpha_D(y)\dashv(z\dashv x).\nonumber
\end{eqnarray}
The left hand side vanishes thanks also to the Definition \ref{dia}. Thus the conclusion holds.
\end{proof}
 
\section{Modules over color Hom-Leibniz algebras}
In this section we introduce modules over color Hom-Leibniz algebras and prove that the Yau's twisting of module structure map 
works very well with module over color Hom-Leibniz algebras.
\begin{definition}
Let $G$ be an abelian group.
 A Hom-module is a pair $(M,\alpha_M)$ in which $M$ is a $G$-graded vector space and $\alpha_M : M\longrightarrow M$ is an even linear map.
\end{definition}
\begin{definition}
Let $(L, [-, -], \varepsilon, \alpha_L)$ be a color Hom-Leibniz algebra and $(M, \alpha_M)$ a Hom-module. An $L$-module on $M$ consists of
even $\bf K$-bilinear maps
$\mu_L : L\otimes M\rightarrow M$ and $\mu_R : M\otimes L\rightarrow M$ such that for any $x, y\in L$, and $m\in M$, 
\begin{eqnarray}
\alpha_M(\mu_L(x, m))&=&\mu_L(\alpha_L(x), \alpha_M(m))\\
\alpha_M(\mu_R(m, x))&=&\mu_R(\alpha_M(m), \alpha_L(x))\\
 \mu_L([-, -]\otimes\alpha_M)&=&\mu_L(\alpha_L\otimes\mu_L)-\varepsilon(x, y)\mu_L(\alpha_L\otimes\mu_L)(\tau_{L, L}\otimes Id_L),\label{lm1}\\
\mu_R(\alpha_M\otimes[-, -])&=&\mu_R(\mu_L\otimes\alpha_L)(\tau_{M, L}\otimes Id_L)
+\varepsilon(x, m)\mu_L(\alpha_L\otimes\mu_R)(\tau_{M, L}\otimes Id_L),\label{lm2}\\
\mu_L(\alpha_L\otimes\mu_R)&=&\mu_R(\mu_L\otimes\alpha_L)+\varepsilon(m, x)\mu_R(\alpha_M\otimes [-, -])(\tau_{L, M}\otimes Id_L)\label{lm3}.
\end{eqnarray}
where $\tau_{L, L}(x\otimes y)=y\otimes x$, $\tau_{L, M}(x\otimes m)=m\otimes x$, $\tau_{M, L}(m\otimes x)=x\otimes m$.
\end{definition}
\begin{remark}
The conditions (\ref{lm1}), (\ref{lm2}) and (\ref{lm3}) can be written respectively as
\begin{eqnarray}
 [x, y]\cdot \alpha_M(m)&=&\alpha_L(x)\cdot(y\cdot m)-\varepsilon(x, y)\alpha_L(y)\cdot(x\cdot m)\label{lm11},\\
\alpha_M(m)\ast[x, y]&=&(x\cdot m)\ast\alpha_L(y)+\varepsilon(x, m)\alpha_L(x)\cdot(m\ast y),\label{lm22}\\
 \alpha_L(x)\cdot(m\ast y) &=&(x\cdot m)\ast\alpha_L(y)+\varepsilon(m, x)\alpha_M(m)\ast[x, y],\label{lm33}
\end{eqnarray}
where $\mu(x\otimes y)=xy$, $\mu_L(x\otimes m)=x\cdot m$ and $\mu_R(m\otimes x)=m\ast x$.
\end{remark}
\begin{example}
 Any color Hom-Leibniz algebra and any Hom-Leibniz superalgebras is a module over itself.
\end{example}
\begin{theorem}\label{mcl}
Let $(L, [-, -], \varepsilon, \alpha_L)$ be a color Hom-Leibniz algebra and $(M, \mu_L, \mu_R, \alpha_M)$ a color Hom-Leibniz module. Then,
\begin{eqnarray}
 \tilde\mu_L&=&\mu_L(\alpha_L^2\otimes Id_M) : L\otimes M\rightarrow M,\\
 \tilde\mu_R&=&\mu_R(Id_M\otimes \alpha_L^2) : L\otimes M\rightarrow M,
\end{eqnarray}
define another color Hom-Leibniz module structure on $M$.
\end{theorem}
\begin{proof}
We have to point out relations (\ref{lm1})-(\ref{lm3}) for $\tilde\mu_L$ and $\tilde\mu_R$. We have, respectively, any for $x, y\in L, m\in M$,
\begin{eqnarray}
\tilde\mu_L([-, -]\otimes\alpha_M)(x\otimes y\otimes m) &=& \tilde\mu_L([x, y]\otimes \alpha_M(m)) \cr 
                                                        &=& \alpha_L^2([x, y])\cdot\alpha_M(m)\nonumber\\
                                                        &=& [\alpha_L^2(x), \alpha_L^2(y)]\cdot\alpha_M(m)\nonumber\\
                                                        &=& \alpha_L^3(x)\cdot(\alpha_L^2(y)\cdot m)-\varepsilon(x, y)
                                                        \alpha_L^3(y)\cdot(\alpha_L^2(x)\cdot m)\; (\mbox{ by } (\ref{lm11}))\nonumber\\
                                                        &=& \tilde\mu_L(\alpha_L\otimes\tilde\mu_L)(x\otimes y\otimes m)\nonumber\\
                                                        & &-\varepsilon(x, y)\tilde\mu_L(\alpha_L\otimes\tilde\mu_L)(\tau_{L, L}
                                                        \otimes Id_L)(x\otimes y\otimes m).\nonumber
 \end{eqnarray}
 \begin{eqnarray}
\tilde\mu_R(\alpha_M\otimes[-, -])(m\otimes x\otimes y)&=&\tilde\mu_R(\alpha_M(m)\otimes [x, y])\nonumber\\
&=&\alpha_M(m)\ast[\alpha_L^2(x), \alpha_L^2(y)]\nonumber\\
&=&(\alpha_L^2(x)\cdot m)\ast\alpha_L^3(y)+\varepsilon(x, m)\alpha_L^3(x)\cdot(m\ast\alpha_L^2(y))\quad(\mbox{by}\;(\ref{lm22}))\nonumber\\
&=&\tilde\mu_R(\tilde\mu_L\otimes\alpha_L)(\tau_{M, L}\otimes Id_L)(m\otimes x\otimes y)\nonumber\\
& &+\varepsilon(x, m)\tilde\mu_L(\alpha_L\otimes\tilde\mu_R)(\tau_{M, L}\otimes Id_L)(m\otimes x\otimes y).\nonumber
 \end{eqnarray}
 \begin{eqnarray}
\varepsilon(m, x)\tilde\mu_R(\alpha_M\otimes[-, -])(\tau_{L, M}\otimes Id_L)(x\otimes m\otimes y)
&=&\varepsilon(m, x)\tilde\mu_R(\alpha_M\otimes[-, -])(x\otimes m\otimes y)\nonumber\\
&=&\varepsilon(m, x)\tilde\mu_R(\alpha_M(m)\otimes [x, y]) \cr 
&=&\varepsilon(m, x)\alpha_M(m)\ast[\alpha_L^2(x), \alpha_L^2(y)]\nonumber\\
&=&\alpha_L^3(x)\cdot(m\ast\alpha_L^2(y)) \cr 
& &-(\alpha_L^2(x)\cdot m)\ast\alpha_L^3(y)\quad(\mbox{by}\;(\ref{lm33})) \nonumber\\
&=&\tilde\mu_L(\alpha_L\otimes\tilde\mu_R)(x\otimes m\otimes y) \cr 
& &-\tilde\mu_R(\tilde\mu_L\otimes \alpha_L)(x\otimes m\otimes y).\nonumber
\end{eqnarray}
\end{proof}
\begin{corollary}
 Let $A_{\alpha_A^n}=(A, [-, -]_{\alpha_A^n}, \varepsilon, \alpha^n_A)$ be a multiplicative color Hom-Leibniz algebra as in Theorem 
\ref{n1} and $(M, \mu_L, \mu_R, \alpha_M)$ an $A_{\alpha_A^n}$-module. Then $(M, \tilde\mu_L, \tilde\mu_R, \alpha_M)$ is also 
a module over $A_{\alpha_A^n}$.
\end{corollary}



%
%
%



\end{document}